\documentclass[reqno,12pt]{amsart}

\usepackage[latin1]{inputenc}

\usepackage{enumerate}
\usepackage{amstext,amsmath,amsthm,amsfonts,amssymb,amscd}
\usepackage{latexsym,mathrsfs,eufrak,dsfont,euscript}
\usepackage{fullpage}
\usepackage{txfonts}
\usepackage{color}

\usepackage{hyperref}

\newtheorem{theorem}{Theorem}
\newtheorem*{thm}{Theorem}
\newtheorem{proposition}[theorem]{Proposition}
\newtheorem{lemma}[theorem]{Lemma}
\newtheorem{corollary}[theorem]{Corollary}
\theoremstyle{definition}

\theoremstyle{remark}

\numberwithin{equation}{section}

\newcommand{\abs}[1]{\left\vert#1\right\vert}
\newcommand{\set}[1]{\left\{#1\right\}}

\newcommand{\norm}[1]{\left\Vert#1\right\Vert}

\allowdisplaybreaks
\begin{document}
\title[Heavy tailed approximate identities]{Heavy tailed approximate identities and $\sigma$-stable Markov kernels}
%
%


\author[]{Hugo Aimar}
\email{haimar@santafe-conicet.gov.ar}
\author[]{Ivana G\'{o}mez}
\email{ivanagomez@santafe-conicet.gov.ar}
\author[]{Federico Morana}
\email{fmorana@santafe-conicet.gov.ar}
\thanks{The research was supported  by CONICET, ANPCyT (MINCyT) and UNL}

\begin{abstract}
The aim of this paper is to present some results relating the properties of stability, concentration and approximation to the identity of convolution through not necessarily mollification type families of heavy tailed Markov kernels. A particular case is provided by the kernels $K_t$ obtained as the $t$ mollification of $L^{\sigma(t)}$ selected from the family $\mathcal{L}=\{L^{\sigma}: \widehat{L^{\sigma}}(\xi)=e^{-\abs{\xi}^\sigma}, 0<\sigma<2\}$, by a given function $\sigma$ with values in the interval $(0,2)$. We show that a basic Harnack type inequality, introduced by C.~Calder\'{o}n in the convolution case, becomes at once natural to the setting and useful to connect the concepts of stability, concentration and approximation of the identity. Some of the general results are extended to spaces of homogeneous type since most of the concepts involved in the theory are given in terms of metric and measure.
\end{abstract}
\maketitle

\section{Introduction}
Let us  start by reviewing some fundamental properties of the family of Poisson (Cauchy) kernels in $\mathbb{R}^n$. For $0<\sigma<2$, set $R^\sigma(\rho)=(1+\rho^2)^{-(n+\sigma)/2}$ and $I(\sigma)=\omega_n\int_0^\infty R^\sigma(\rho)\rho^{n-1} d\rho$, with $\omega_n$ the surface area of the unit sphere in $\mathbb{R}^n$. Since $I(\sigma)\geq \omega_n 2^{-(n+\sigma)/2}\int_1^\infty \rho^{n-1}\rho^{-n-\sigma} d\rho=\omega_n \sigma^{-1} 2^{-(n+\sigma)/2}$, we have that $I(\sigma)\to +\infty$ when $\sigma\to 0$. Set $P^\sigma$ to denote the Cauchy-Poisson kernel of order $\sigma \in (0,2)$, defined on $\mathbb{R}^n$ by $P^\sigma(x)=(I(\sigma))^{-1}R^\sigma(\abs{x})=(I(\sigma))^{-1}(1+\abs{x}^2)^{-(n+\sigma)/2}$. Let us observe that for fixed $\sigma$, $\int_{\mathbb{R}^{n}}P^\sigma(x) dx=1$ and for fixed $x\in \mathbb{R}^n$, $P^\sigma(x)\to 0$, $\sigma\to 0$. Let us now take the mollification $u(x,y;\sigma)$, of any $P^\sigma$, for $y>0$. Precisely, $u(x,y;\sigma)=y^{-n}P^\sigma(y^{-1}x)=(I(\sigma))^{-1} y^{\sigma}(y^2+\abs{x}^2)^{-(n+\sigma)/2}$. Let $\mathcal{P}$ be the two parametric family of kernels $\mathcal{P}=\{u(\cdot,y;\sigma):y>0, 0<\sigma<2\}$. Some well known properties of the members of $\mathcal{P}$ are, $\int_{\mathbb{R}^n}u(x,y;\sigma)dx=1$ for every $y>0$
 and every $\sigma\in (0,2)$; $div(y^{1-\sigma}grad\ u_\sigma)=0$, where $u_\sigma(x,y)=u(x,y;\sigma)$ and the operators $div$ and $grad$ are computed in $R^{n+1}_+=\{(x,y): x\in \mathbb{R}^n, y>0\}$ (see \cite{CaSi07}) and  $\lim_{y\to 0} u_\sigma(x,y)$ is the Dirac delta for fixed $\sigma$.
 For our later analysis it would be also important  to observe the behavior of $u(x,y;\sigma)$ for $\abs{x}$ large. It is easy to prove that  $\abs{x}^{n+\sigma}u(x,y;\sigma)\to (I(\sigma))^{-1} y^{\sigma}$ as $\abs{x}\to +\infty$. The family $\mathcal{P}$ satisfies a uniform Harnack type inequality in $\mathbb{R}^n\setminus \{0\}$. Let $0<\gamma<1$ be given. Let us compare the infimum and the supremum of any member of $\mathcal{P}$ in balls of the form $B(x,\gamma\abs{x})$ with $x\in \mathbb{R}^n\setminus\{0\}$. Notice first that $\mathfrak{s}=\sup_{z\in B(x,\gamma\abs{x})}u(z,y;\sigma)=(I(\sigma))^{-1} y^\sigma (y^2+\abs{x}^2(1-\gamma)^2)^{-(n+\sigma)/2}$ and $\mathfrak{i}=\inf_{z\in B(x,\gamma\abs{x})}u(z,y;\sigma)=(I(\sigma))^{-1} y^\sigma (y^2+\abs{x}^2(1+\gamma)^2)^{-(n+\sigma)/2}$. Hence  $\mathfrak{s}\leq ((1+\gamma)/(1-\gamma))^{n+\sigma} \mathfrak{i}\leq ((1+\gamma)/(1-\gamma))^{n+2} \mathfrak{i}$, for every $x\in \mathbb{R}^n$, every $y>0$ and every $0<\sigma<2$, see \S~3. Let us observe that this type of inequality is impossible for the Gaussian family, which corresponds to the case $\sigma=2$.
 For fixed $\sigma$ and $y$ approaching $0$, the functions of $x$, $u(x,y;\sigma)$ tends to the Dirac delta, because the integral in $x$ is preserved and they concentrate the mass around the origin. On the other hand, for fixed $y$, the functions of $x$ given by $u(x,y;\sigma)$ tend to zero as $\sigma\to 0^+$ also preserving the integral with respect to $x$. As we shall prove in Section~2, the maximal operator $\sup_{y>0,\, 0<\sigma<2}\abs{u(\cdot,y;\sigma)\ast f}$ is of weak type (1,1). Actually this operator is bounded pointwise by the Hardy-Littlewood function $Mf(x)$. As a corollary we shall prove that if $\sigma(y)\log y\to-\infty$ as $y\to 0^+$, then $u(\cdot,y;\sigma(y))$ concentrates around the origin when $y\to 0^+$. Hence for a function $\sigma: \mathbb{R}^+\to (0,2)$ for which both $\sigma(y)\to 0^+$ and $\sigma(y)\log y\to -\infty$ when $y\to 0^+$, the kernel $K_t(x)=u(x,y;\sigma(y))$ provides a pointwise approximate identity. More difficult is the same problem for the L\'{e}vy family $\mathcal{L}=\{v(x,y;\sigma): \widehat{v}(\xi,y;\sigma)=e^{-\abs{y\xi}^\sigma}, 0<\sigma<2\}$. For the one dimensional case we prove the weak type (1,1) of the operator $\sup_{y>0,\, 0<\sigma<2}\abs{v(\cdot,y;\sigma)\ast f}$ by using the Theorem of Z\'{o} (see \cite{Zo76} and \cite{DeGuzman81}) and some basic formulas for the Bessel functions.

 After that we start the consideration of a general theory relating the concepts of stability, concentration, approximate identities and Harnack type inequalities. Let us introduce heuristically the basic ideas of the concepts of stability, concentration, approximation of the identity and Harnack type inequality that we shall formally define in Section 2.

\textit{Stability.}
The word stability comes from probability theory. A random variable $X$ is said to be stable, if for every choice of positive real constants $A$ and $B$ there exist $C>0$ and $D\in \mathbb{R}$ and two independent random variables $X_i$, $i=1,2$, with the same distribution that $X$ such that $AX_1+BX_2\stackrel{d}{=}CX+D$. Here $\stackrel{d}{=}$ means equally distributed. In terms of characteristic functions (Fourier transform of distributions) when the distribution has a density $g$, the stability is equivalent to $\hat{g}(A\xi)\hat{g}(B\xi)=\hat{g}(C\xi+D)$. The basic example of densities satisfying this property is the one parameter family $\mathcal{L}=\{g^\sigma: 0<\sigma\leq 2\}$ whose Fourier transforms are given by $\widehat{g^\sigma}(\xi)=e^{-\abs{\xi}^{\sigma}}$. The stability condition is satisfied by taking $D=0$ and $C=(A^\sigma+B^\sigma)^{1/\sigma}$. Two values of $\sigma$ provide the most classical examples. When $\sigma=2$, $g^2$ is the Gaussian. When $\sigma=1$, $g^1$ is the Cauchy (Poisson) distribution. While $g^2$ has finite moments of all order, $g^1$ has not even a finite first order moment. In \cite{BluGe60} it is proved that for every $0<\sigma<2$ the behavior at infinity of $g^\sigma$ is of the order of $\abs{x}^{-n-\sigma}$. Actually, Blumenthal and Getoor show that the limit as $\abs{x}\to\infty$ of $\abs{x}^{n+\sigma}g^\sigma(x)$ is a positive real number. The variance of a random variable distributed according to a radial density $g$ defined on $\mathbb{R}^n$ can be written as
$\textrm{Var}g=\int_{\mathbb{R}^n}g(\abs{x})\abs{x}^2 dx=\omega_n\int_{0}^\infty g(\rho)\rho^{n+2-1}d\rho
=\omega_n\int_{0}^{1} g(\rho)\rho^{n+1}d\rho + \omega_n\int_{1}^{\infty}\rho^{n+2} g(\rho)\frac{d\rho}{\rho}$.
Here $\omega_n$ denotes the area measure of the unit sphere in $\mathbb{R}^n$. Hence the variance of $g$ is finite if and only if $\rho^{n+2}g(\rho)$ is integrable on the half line $(1,\infty)$ with respect to $\tfrac{d\rho}{\rho}$, the Haar measure in $\mathbb{R}^+$. In particular if $\rho^{n+2}g(\rho)$ is supposed to converge when $\rho\to\infty$, the limit has to  vanish when the variance is finite.
Notice that if for some $0<\sigma<2$, $\rho^{n+\sigma}g(\rho)$ has a positive limit when $\rho\to\infty$, then the variance of the random variable $X$ with density $g$ is infinite.
A deeper relation between the variance and $\Lambda_\sigma(g)=\lim_{\rho\to\infty}\rho^{n+\sigma}g(\rho)$ can be observed when dealing with the mollification $g_t(x)=t^{-n}g(t^{-1}x)$ which determines the distribution of the random vector $tX$. In fact, if $\textrm{Var}g=1$, then $\textrm{Var}g_t=t^2$. On the other hand, if $\Lambda_\sigma(g)=1$, the $\Lambda_\sigma(g_t)=t^\sigma$. In particular, when $g$ is the Cauchy-Poisson kernel $g^1(x)= c_n(1+\abs{x}^2)^{-(n+1)/2}$, we have that $\Lambda_1(g_t)=t$. More generally, from the result of Blumenthal and Getoor \cite{BluGe60}, when $g$ is the inverse Fourier transform of $e^{-\abs{\xi}^\sigma}$, $0<\sigma<2$, we have that $\Lambda_\sigma(g_t)=t^\sigma$.

In this paper the expression $\sigma$-stable shall be applied to Markov (non-convolution) kernels to reflect their behavior at infinity in terms of limits of type $\Lambda_\sigma$. The precise definitions shall be given in Section 2.

\textit{Concentration.}
When a random vector $X$ in $\mathbb{R}^n$ has finite variance and it is distributed according to the density $g$, the family $X_t(\omega)=tX(\omega)$, $t>0$ concentrates around the origin of $\mathbb{R}^n$ as $t$ tends to zero. On the distribution side, the density $g_t$ of $X_t$ is given by the mollification of $g$: $g_t(x)=t^{-n}g(t^{-1}x)$. With $\int_{\mathbb{R}^n}g_tdx=1$ for every $t>0$. Hence $g_t$ the density of $X_t$ approximates the Dirac delta when $X_t$ concentrates around the origin. By Chebyshev inequality, an estimate for the rate of concentration of $X_t$ is given by $\mathscr{P}(\{\abs{X_t}\geq\lambda\})\leq\tfrac{t^2}{\lambda^2}\textrm{Var}(X)$ for every $\lambda>0$. So that for fixed $\lambda>0$, by taking $t$ small we get concentration inside the ball centered at the origin with radious $\lambda$.

We have seen how $\sigma$-stability of random vectors may be regarded as an extension of the variance for heavy tailed distributions. Then, a natural question is whether or not concentration follows from stability. The first remark regarding this question is that $\sigma$-stability is not a global property of a density, as is the finitness of the variance. In the classical heavy tailed Cauchy distribution $g^1_t(x)=c_n t(t^2+\abs{x}^2)^{-(n+1)/2}$ we have that $\mathscr{P}(\{\abs{X_t}\geq\lambda\})=\int_{\abs{x}\geq\lambda/t}g(x)dx$, which tends to zero as $t\to 0$ for every $\lambda>0$. Hence some heavy tailed stable distributions have the concentration property.

On the other hand, concentration of densities is an important item to produce good approximations of the Dirac delta.

\textit{Approximate identities.}
Let $\{X_j: j\in \mathbb{N}\}$ be a sequence of independent and identically distributed random vectors in $\mathbb{R}^n$ with finite variance and mean zero. Then, from the weak law of large numbers, $\mathscr{P}\{\abs{N^{-1}\sum_{j=1}^{N}X_j}>\lambda\}\to 0$ as $N$ tends to infinity. Assuming that $X_1$ has a density $g$, the weak law implies that the kernels defined in $\mathbb{R}^n$ by the mollification of the $N$ fold convolution of $g$, $G_N(x)=N^n(g\ast\ldots\ast g)(Nx)$ has the concentration property inside the ball $B(0,\lambda)$ as $N\to\infty$. Hence, at least in a weak sense $\{G_N: N\in \mathbb{N}\}$ is an approximation of the Dirac delta, the identity of convolution. The problem of the pointwise approximation of the identity of convolution, is related, not only to the concentration property but also with some particular regularities on the family of kernels that concentrates around the origin, which allow the weak boundedness of the associated maximal operator. Let us precise the above and at the same time comment on some basic well known result of harmonic analysis related to approximating the identity.
The most classical mollification type  approximation of the identity of convolution in $\mathbb{R}^n$ is given by $\varphi_r(x)=r^{-n}\varphi(r^{-1}x)$ for $r>0$, with $\varphi(x)=\abs{B(0,1)}^{-1}\mathcal{X}_{B(0,1)}(x)$. This case corresponds to the Lebesgue Differentiation Theorem, where concentration involves compact support in balls around the origin with radii that become small when $r$ tends to zero. In this case the associated boundedness property corresponds to the Hardy--Littlewood Maximal Theorem. A second instance for $\varphi$, which allows many interesting and useful situations, is the case when $\abs{\varphi}$ admits an integrable and decreasing radial majorant. In this case, even when usually the support of no $\varphi_r$ is compact, the residual mass outside any ball centered at the origin becomes as small as desired taking $r$ small enough. Moreover, the Hardy--Littlewood maximal function has still the upper control of the boundedness properties. There are some other settings where the above results extend such as the nonisotropic case of parabolic metrics. Nevertheless all these situations are particular cases of a deep
result due to F.~Z\'{o} \cite{Zo76} (see also \cite{DeGuzman81} chapter~10), which follows the Calder\'{o}n--Zygmund approach to singular integrals and can also be treated from the vector valued point of view of the Calder\'{o}n--Zygmund theory.

\textit{Harnack type inequality.}
In the regularity theory of solutions of elliptic and parabolic equations due to De~Giorgi, Nash, and Moser, Harnack's inequality plays a central role.
As we have seen, for the family $\mathcal{P}$ of Poisson kernels, we have that
\begin{equation}\label{eq:harnackorigin}
\sup_{z\in B(x,\gamma\abs{x})}u(z,y;\sigma)\leq \left(\frac{1+\gamma}{1-\gamma}\right)^{n+\sigma}\inf_{z\in B(x,\gamma\abs{x})}u(z,y;\sigma).
\end{equation}
Actually \eqref{eq:harnackorigin} was used by C.~Calder\'{o}n as a sufficient condition for the pointwise approximation to the identity of convolution. This condition was orally communicated in the early eighties of the last century by Prof. Calixto Calder\'{o}n to the first author of this article. For a related condition, see \cite{CalUr15}.

In Section~\ref{sec:CauchyLevy} we prove that two parametric maximal functions
\begin{equation*}
P^*f(x)=\sup_{y>0,\, 0<\sigma<2}\abs{P^\sigma_y\ast f}
\end{equation*}
and, for $n=1$,
\begin{equation*}
L^*f(x)=\sup_{y>0,\, 0<\sigma<2}\abs{L^\sigma_y\ast f}
\end{equation*}
are weak type (1,1) operators. We also give some sufficient conditions on $\sigma: \mathbb{R}^+\to (0,2)$ such that $\{P^{\sigma(y)}_y: y>0\}$ and $\{L^{\sigma(y)}_y: y>0\}$ ($n=1$) are pointwise approximate identities. Section~\ref{sec:StableHarnackEuclidean} is devoted to introduce the general setting in $\mathbb{R}^n$ and to state the main results relating stability, concentration and approximation of the identity through Harnack inequalities. Finally in Section~\ref{sec:generalsetting} we state and prove some extension to spaces of homogeneous type of the results in Section~\ref{sec:StableHarnackEuclidean}.

\section{Cauchy and L\'{e}vy distributions}\label{sec:CauchyLevy}
The first part of this section is devoted to the analysis of the boundedness of the maximal operator induced by the two parametric family of Cauchy-Poisson kernels $\mathcal{P}=\{P^\sigma_y: 0<\sigma<2, y>0\}$. Here, as in the introduction, $P^\sigma_y(x)=(I(\sigma))^{-1}y^{\sigma}(y^2+\abs{x}^2)^{-(n+\sigma)/2}$, $0<\sigma<2$, $y>0$, $x\in \mathbb{R}^n$. We also give a sufficient condition on an \textit{stability order selection function} $\Sigma: \mathbb{R}^+\to (0,2)$ in such a way that the one parameter family $\mathcal{P}_{\Sigma}=\{P^{\Sigma(y)}_y: y>0\}$ concentrates around the origin of $\mathbb{R}^n$ and, as a consequence, $(P^{\Sigma(y)}_y\ast f)(x)$ tends to $f(x)$ as $y\to 0^+$ for almost every $x\in \mathbb{R}^n$, and every function $f\in L^p(\mathbb{R}^n)$, $1\leq p\leq\infty$. The sufficient  condition on $\Sigma(y)$ that we give in order to have concentration, allows functions $\Sigma$  for which $\Sigma(y)$ tends to zero when $y\to 0^+$. For fixed $\sigma$ and $y$ tending to zero, we certainly have concentration. On the other hand, for fixed $y>0$ and $\sigma\to 0^+$ we have dissipation instead of concentration, i.e. $P^\sigma_y\to 0$ when $\sigma\to 0^+$ for fixed $y$. Our condition in $\Sigma$ can be regarded as a control in the way in which we can select distributions from the Cauchy family with orders as small as desired in terms of the mollification parameter $y$. In the second part of this section we aim to solve the same problem when, instead of $\mathcal{P}$, we have the class $\mathcal{L}$ of L\'{e}vy
 kernels, $\mathcal{L}=\{L^\sigma_y: 0<\sigma<2, y>0\}$ and $L^\sigma_y$ is the mollification $L^\sigma_y(x)=y^{-n}L^\sigma(y^{-1}x)$ with $L^\sigma$ the density whose characteristic function is given by $\widehat{L^\sigma}(\xi)=e^{-\abs{\xi}^{\sigma}}$.

At a first glance since, using the results of Blumenthal and Getoor in \cite{BluGe60}, the kernels $L^\sigma$ are bounded above by a constant times the corresponding $P^\sigma$, the second result seems to reduce  to the first. Nevertheless the constant $C$ such that $L^\sigma(x)\leq CP^\sigma(x)$ that we obtain from \cite{BluGe60} depends on $\sigma$ in an unknown way when $\sigma\to 0^+$. To skip this situation we use the theorem of Z\'{o} (\cite{Zo76}, \cite{DeGuzman81}) in order to get the weak type (1,1) of the two parameter maximal operator $L^*f(x)=\sup_{y>0, 0<\sigma<2}\abs{(L^\sigma_y\ast f)(x)}$ in dimension $n=1$.

Let us start by stating and proving the result for the Cauchy--Poisson family. As usual we shall denote by $M$ the Hardy--Littlewood maximal operator defined for every locally integrable function $f$ by
\begin{equation*} Mf(x)=\sup_{r>0}\frac{1}{\abs{B(x,r)}}\int_{B(x,r)}\abs{f(z)}dz\end{equation*}
This operator satisfies the weak type (1,1) inequality $\abs{\{x\in \mathbb{R}^n: Mf(x)>\lambda\}}\leq C\lambda^{-1}\norm{f}_1$ for some constant $C$ which depends on dimension, for every $\lambda>0$ and every $f\in L^1(\mathbb{R}^n)$. The boundedness of $M$ in $L^\infty(\mathbb{R}^n)$ and the Marcinkiewicz interpolation argument show also that $M$ is bounded as an operator defined in $L^p(\mathbb{R}^n)$, $1<p\leq\infty$. In other words, $\norm{Mf}_p\leq C\norm{f}_p$.
\begin{theorem}\label{thm:maximaloperatorforP}\quad
\begin{enumerate}[(a)]
\item  The operator $P^*f(x)=\sup_{y>0, 0<\sigma<2}\abs{(P^\sigma_y\ast f)(x)}$ is bounded above by a constant times the Hardy--Littlewood maximal operator. Hence $P^*$ is of weak type (1,1) and bounded on $L^p(\mathbb{R}^n)$ for $1<p\leq\infty$.
\item For a function $\Sigma: \mathbb{R}^+\to (0,2)$ such that $\Sigma(y)\log y\to -\infty$ when $y\to 0^+$, the one para\-me\-ter family of kernels $\{P^{\Sigma(y)}_y: y>0\}$ concentrates about the origin, i.e. for every $\lambda>0$ $\int_{\abs{x}>\lambda}P^{\Sigma(y)}_y(x) dx\to 0$ when $y\to 0^+$.
\item With $\Sigma$ as in (b), we have that $(P^{\Sigma(y)}_y\ast f)(x)$ converges almost everywhere to $f(x)$ for $f\in L^p(\mathbb{R}^n)$, $1\leq p\leq\infty$.
\end{enumerate}
\end{theorem}
\begin{proof}
\textit{(a)} Follows from the facts that each $P^\sigma$ is radial, decreasing and $\int P^\sigma dx=1$. See for example \cite[Theorem~2,\S2, Chapter~III]{Stein70}.
\textit{(b)} We have to prove that if $y^{\Sigma(y)}\to 0$ as $y\to 0^+$, then $\{P^{\Sigma(y)}_y:y>0\}$ concentrates around the origin when $y\to 0^+$. Let $\lambda>0$ be given, then
\begin{align*}
\int_{\abs{x}\geq\lambda}P^{\Sigma(y)}_y(x) dx &= [I(\Sigma(y))]^{-1} y^{\Sigma(y)}\int_{\abs{x}\geq\lambda}(y^2+\abs{x}^2)^{-\tfrac{n+\Sigma(y)}{2}}dx\\
&= [I(\Sigma(y))]^{-1}\int_{\abs{x}\geq\tfrac{\lambda}{y}}(1+\abs{x}^2)^{-\tfrac{n+\Sigma(y)}{2}}dx\\
&\leq \left[\omega_n (\Sigma(y))^{-1}2^{-\tfrac{n+\Sigma(y)}{2}}\right]^{-1}\omega_n\int_{\tfrac{\lambda}{y}}^{\infty}\frac{\rho^{n-1}}{(1+\rho^2)^{-(n+\Sigma(y))/2}}\\
&\leq \Sigma(y)2^{\tfrac{n+\Sigma(y)}{2}}\int_{\tfrac{\lambda}{y}}^{\infty}\rho^{-1-\Sigma(y)}d\rho\\
&=2^{\tfrac{n+\Sigma(y)}{2}}\left(\frac{\lambda}{y}\right)^{-\Sigma(y)}\\
&\leq C(n,\lambda) y^{\Sigma(y)}
\end{align*}
which tends to zero for $y\to 0^+$. Now, if $g$ is a compactly supported continuous function we have for a given $\varepsilon>0$ that there exists $\delta>0$ such that
\begin{align*}
\abs{(P^{\Sigma(y)}_y\ast g)(x)-g(x)}&\leq \int_{\mathbb{R}^n}\abs{g(x-z)-g(x)}P^{\Sigma(y)}_y(z)dz\\
&\leq\varepsilon\int_{\abs{z}<\delta}P^{\Sigma(y)}_y(z) dz+2\norm{g}_\infty\int_{\abs{z}\geq\delta}P^{\Sigma(y)}_y(z) dz,
\end{align*}
which is less than $2\varepsilon$ for $y$ small enough. Hence, since $P^*_\Sigma f(x)=\sup_{y>0}\abs{(P^{\Sigma(y)}_y\ast f)(x)}\leq P^*f(x)$,  standard arguments can be applied to show that $(P^{\Sigma(y)}_y\ast f)(x)\to f(x)$ for $y\to 0^+$ for almost every $x$ when $f\in L^p(\mathbb{R}^n)$, $1\leq p\leq\infty$.
\end{proof}
Aside for the case of $\Sigma$ bounded below by a positive constant, an example of such a $\Sigma$ is given by $\Sigma(y)=(\log^\alpha y^{-1})^{-1}$ for $0<\alpha<1$. In this case $\Sigma(y)\to 0$ when $y\to 0$.

Let us  now move to the L\'{e}vy kernel family $\mathcal{L}=\{L^\sigma_y: \widehat{L^\sigma}(\xi)=e^{-\abs{\xi}^\sigma}; 0<\sigma<2; L^\sigma_y(x)=y^{-n}L^\sigma(y^{-1}x), y>0\}$. Let us recall some known results regarding this family of kernels in one dimension. Basic references are the work Blumenthal and Getoor \cite{BluGe60}, Polya \cite{Pol23} and the formulae for the Fourier Transform of radial functions given in \cite{BoChan49}. Set, as we did in the Introduction, $v(x,y;\sigma)$ to denote the mollification by the parameter $y>0$ of the inverse Fourier transform $v(x,1;\sigma)$ of $e^{-\abs{\xi}^{\sigma}}$. Precisely, in dimension $n=1$,
\begin{equation}\label{eq:kernelBluGetone}
v(x,1;\sigma)=\sqrt{2\pi\abs{x}}\int_0^\infty e^{-t^\sigma}t^{1/2}\mathcal{J}_{-\tfrac{1}{2}}(\abs{x}t)dt,
\end{equation}
where $\mathcal{J}_{-\tfrac{1}{2}}$ denotes the Bessel function of the first kind of order $-\tfrac{1}{2}$. The result in \cite{BluGe60} provides the asymptotic behavior of $v(x,1;\sigma)$ which shows the stability of these kernels
\begin{equation*}
\lim_{\abs{x}\to\infty}\abs{x}^{1+\sigma} v(x,1,\sigma)=\sigma 2^{\sigma -1}\pi^{-\tfrac{3}{2}} \sin\frac{\sigma\pi}{2}\Gamma\left(\tfrac{1+\sigma}{2}\right)\Gamma\left(\tfrac{\sigma}{2}\right).
\end{equation*}
Since $e^{-\abs{\xi}^\sigma}$ belongs to $L^1(\mathbb{R})$, we also have that $v(x,1;\sigma)$ belongs to $L^\infty(\mathbb{R})$. Hence we have bounds of the form $v(x,y;\sigma)\leq C P^\sigma_y(x)$. But since the constant $C$ depends on $\sigma$, this estimate is not sharp in order to get upper estimates of $L^*f(x)=\sup_{y>0, 0<\sigma<2}\abs{(v(\cdot,y;\sigma)\ast f)(x)}$ in terms of the Hardy--Littlewood maximal function of $f$. Instead we shall obtain the weak type (1,1) of $L^*f$ as an application of the theorem of Z\'{o} which we proceed to state as is given in \cite{DeGuzman81}. There, the theorem is proved for not necessarily mollification type family of kernels. In our further analysis, aside from \eqref{eq:kernelBluGetone}, we shall deal with $v(x,1;\sigma)$ in three dimensions. In general, for dimension $n$ the basic kernel is given by
\begin{equation}\label{eq:kernelBluGetgeneral}
v(x,1;\sigma)=\frac{1}{(2\pi)^{\tfrac{n}{2}}}\frac{1}{\abs{x}^{\tfrac{n-2}{2}}}\int_0^\infty e^{-t^\sigma}t^{n/2}\mathcal{J}_{\tfrac{n-2}{2}}(\abs{x}t)dt.
\end{equation}
Let us recall that an operator $T$ is said to be of weak type (1,1) if there exists a constant $C>0$ such that for every $f\in L^1(\mathbb{R}^n)$ the inequality $\abs{\{x\in \mathbb{R}^n:\abs{Tf(x)}>\lambda\}}\leq\tfrac{C}{\lambda}\norm{f}_1$ holds for every $\lambda>0$. The operator $T$ is said to be bounded on $L^p(\mathbb{R}^n)$ if for some constant $C$ the inequality $\norm{Tf}_p\leq C\norm{f}_p$ holds for every $f\in L^p(\mathbb{R}^n)$.
\begin{thm}[\textbf{Z\'{o}}, 10.3.1 in \cite{DeGuzman81}] \label{thm:theoremZo}
Let $A$ be an index family. Assume that for each $\alpha$ in $A$ an integrable function $k_\alpha$ in $\mathbb{R}^n$ is given satisfying
\begin{enumerate}[(a)]
\item $\int_{\mathbb{R}^n}\abs{k_\alpha (x)}dx$ is bounded above uniformly for $\alpha\in A$;
\item $\int_{\abs{x}\geq 2\abs{z}}F(x,z)dx$ is bounded above uniformly in $z\in \mathbb{R}^n\setminus \{0\}$ with
\begin{equation*}
F(x,z)=\sup_{\alpha\in A}\abs{k_\alpha(x-z)-k_\alpha(x)}.
\end{equation*}
\end{enumerate}
Then the maximal operator $\sup_{\alpha\in A}\abs{k_\alpha \ast f}$ is of weak type (1,1) and bounded on $L^p$ for $1<p\leq\infty$.
\end{thm}
We are now in position to prove the main results concerning the operator $L^*$.
\begin{theorem}\label{thm:applyingZo}
For $n=1$ the operator $L^*$ is of weak type (1,1) and bounded on $L^p(\mathbb{R})$ for $1<p\leq\infty$.
\end{theorem}
\begin{proof}
Let $A=\{\alpha=(\sigma,y):0<\sigma<2, y>0\}$ and for each $\alpha=(\sigma,y)\in A$ set $k_\alpha(x)=v(x,y;\sigma)$. Hypothesis \textit{(a)} in Z\'{o}'s theorem follows from $\int_{\mathbb{R}}k_\alpha(x)dx=\widehat{k_\alpha}(0)=1$ for every $\alpha\in A$. In order to prove \textit{(b)} in the statement above, let us  start by using the formulas contained in \cite[Chapter~II, page~69]{BoChan49}. Define
\begin{equation*}
\varphi_\sigma(s)=\int_0^\infty e^{-t^\sigma}\cos(\sqrt{s}t)dt
\end{equation*}
and
\begin{equation*}
\psi_\sigma(s)=\int_0^\infty e^{-t^\sigma}t\,\mathcal{J}_0(\sqrt{s}t)dt.
\end{equation*}
On the other hand if $\Phi^n_\sigma(\tau)$, $\tau\in \mathbb{R}^+$ is the function such that
$\Phi^n_\sigma(\abs{x})$ is the Fourier transform in $\mathbb{R}^n$ of $e^{-\abs{\xi}^\sigma}$, we have
\begin{equation}\label{eq:Phidimone}
\Phi^n_\sigma(\rho)=(-1)^m\pi^m 2^{2m+1}\frac{d^m\varphi_\sigma}{ds^m}(\rho^2),\quad \textrm{ when } n=2m+1, m=0,1,2,\ldots
\end{equation}
and
\begin{equation*}
\Phi^n_\sigma(\rho)=(-1)^m\pi^{2m+1} 2^{2m+1}\frac{d^m\psi_\sigma}{ds^m}(\rho^2),\quad \textrm{ when } n=2m+2, m=0,1,2,\ldots
\end{equation*}
The one dimensional case $n=1$ corresponds to \eqref{eq:Phidimone} with $m=0$. That is $\Phi^1_\sigma(\rho)=2\varphi_\sigma(\rho^2)$. So
that
\begin{equation*}
\frac{d\Phi^1_\sigma}{d\rho}(\rho) = 4\rho \frac{d\varphi_\sigma}{ds}(\rho^2)=-\frac{\rho}{2\pi}\Phi^3_\sigma(\rho).
\end{equation*}
In other words
\begin{equation}\label{eq:estimatePhi1}
2\pi\rho^2\abs{\frac{d\Phi^1_\sigma}{d\rho}}\leq\rho^3\abs{\Phi^3_\sigma(\rho)}.
\end{equation}
Next we shall show that $\rho^3\abs{\Phi^3_\sigma(\rho)}$ is bounded above by a constant which does not depend on $\sigma\in (0,2)$.
From \eqref{eq:kernelBluGetgeneral} with $n=3$  and the fact that $\mathcal{J}_{\tfrac{1}{2}}(\rho)=\sqrt{2}(\sqrt{\pi\rho})^{-1}\sin\rho$ we see that
\begin{equation*}
\Phi^3(\rho)=(2\pi)^{-3/2}\rho^{-1/2}\int_0^\infty e^{-t^\sigma}t^{3/2}\mathcal{J}_{\tfrac{1}{2}}(\rho t)dt=
\frac{1}{2\pi^2}\rho^{-1}\int_0^\infty e^{-t^\sigma} t\sin(\rho t)dt.
\end{equation*}
Let us show $2\pi^2\rho^3\Phi^3_\sigma(\rho)=\rho^2\int_0^\infty e^{-t^\sigma}t\sin(\rho t)dt$ is bounded above by a constant
which is independent of $0<\sigma<2$ and $\rho>0$. Let $\eta(s)=s\sin s$, $s>0$. Integrating by parts we get
\begin{align*}
\rho^2\int_0^\infty e^{-t^\sigma} t\sin(\rho t)dt &=\rho\int_0^\infty e^{-t^\sigma}\eta(\rho t) dt\\
&=\rho\left[\left.\frac{\sin\rho t-\rho t\cos\rho t}{\rho}e^{-t^\sigma}\right|^\infty_0 + \sigma\int_0^\infty t^{\sigma-1}e^{-t^\sigma}\frac{\sin \rho t-\rho t\cos\rho t}{\rho}\,dt\right]\\
&= \sigma\int_0^\infty t^{\sigma-1}e^{-t^\sigma}\sin\rho t\,dt-\sigma\rho^{1-\sigma}\int_0^\infty e^{-t^\sigma}(\rho t)^\sigma\cos(\rho t)\,dt\\
&=\sigma\int_0^\infty t^{\sigma-1} e^{-t^\sigma}\sin\rho t\, dt-\sigma\rho^{1-\sigma}\left[\left(e^{-t^\sigma}\left.\frac{\zeta(\rho t)}{\rho}\right|^\infty_0\right) + \sigma\int_0^\infty t^{\sigma-1}e^{-t^\sigma}\frac{\zeta(\rho t)}{\rho}dt\right]
\end{align*}
where $\zeta(t)=\int_0^t s^\sigma\cos s\, ds$. Integrating by parts once again
\begin{equation*}
\zeta(t)=\int_0^t s^\sigma\cos s\,ds=\left. s^\sigma\sin s\right|_0^t-\sigma\int^t_0s^{\sigma-1}\sin s\, ds=t^\sigma\sin t-\sigma\int^t_0 s^{\sigma-1}\sin s\, ds
\end{equation*}
Hence
\begin{align*}
\rho^2\int_0^\infty e^{-t^\sigma} t\sin(\rho t)dt&=
\sigma\int_0^\infty t^{\sigma-1}e^{-t^\sigma}\sin\rho t\,dt\\
&\phantom{=}-\sigma^2\rho^{-\sigma}\int_0^\infty t^{\sigma-1}e^{-t^\sigma}(\rho t)^\sigma\sin(\rho t)dt\\
&\phantom{=}
+\sigma^3\rho^{-\sigma}\int_0^\infty t^{\sigma-1}e^{-t^\sigma}\left(\int_0^{\rho t}s^{\sigma-1}\sin s\,ds\right)dt\\
&:= I + II + III
\end{align*}
Changing variables $u=(\rho t)^\sigma$, $I$ and $II$ are given by
\begin{equation*}
I=\int_0^\infty e^{-\tfrac{u}{\rho^\sigma}}\sin(u^{1/\sigma})d\left(\frac{u}{\rho^\sigma}\right)
\end{equation*}
and
\begin{equation*}
II=-\sigma\int_0^\infty e^{-\tfrac{u}{\rho^\sigma}} \frac{u}{\rho^{\sigma}}\sin(u^{1/\sigma})d\left(\frac{u}{\rho^\sigma}\right),
\end{equation*}
both of them are uniformly bounded. For the third term
\begin{align*}
\abs{III}&=\abs{\sigma^2\rho^{-\sigma}\int_0^\infty e^{-\tfrac{u}{\rho^\sigma}}\left(\int_0^{u^{1/\sigma}}s^{\sigma-1}\sin s\, ds\right)d\left(\frac{u}{\rho^\sigma}\right)}\\
&\leq \sigma^2\frac{1}{\rho^\sigma}\int_0^\infty e^{-\tfrac{u}{\rho^\sigma}}\frac{1}{\sigma}(u^{1/\sigma})^{\sigma}d\left(\frac{u}{\rho^\sigma}\right)\\
&=\sigma\int_0^\infty e^{-\tfrac{u}{\rho^\sigma}}\frac{u}{\rho^\sigma}d\left(\frac{u}{\rho^\sigma}\right),
\end{align*}
which is again uniformly bounded above. Hence from \eqref{eq:estimatePhi1} we obtain that $\abs{\frac{d\Phi^1_\sigma}{d\rho}}\leq C\rho^{-2}$ with $C$ independent of $\sigma$ and $\rho$.

We are in position to verify that the family $\{k_\alpha:\alpha\in A\}$ satisfies \textit{(b)} in the theorem of Z\'{o}. In fact, for $\abs{x}\geq 2\abs{z}$ we have, using the mean value theorem that
\begin{equation*}
\abs{v(x-z,y;\sigma)-v(x,y;\sigma)}=\frac{1}{y}\abs{v\Bigl(\frac{x-z}{y},1;\sigma\Bigr)-v\Bigl(\frac{x}{y},1;\sigma\Bigr)}
=\frac{1}{y}\abs{\frac{\partial v}{\partial x}(\xi,1;\sigma)}\frac{\abs{z}}{y}
\end{equation*}
 for some $\xi$ in the segment joining $x/y$ with $(x-z)/y$. From the above estimate for $d\Phi^1_\sigma/d\rho$, since $\abs{\xi}\simeq\abs{x}/y$, because $\abs{x}\geq 2\abs{z}$, we have
\begin{equation*}
\abs{v(x-z,y;\sigma)-v(x,y;\sigma)}\leq C\frac{\abs{z}}{y^2}\left(\frac{y}{\abs{x}}\right)^2.
\end{equation*}
Then $F(x,z)\leq C\frac{\abs{z}}{\abs{x}^2}$ for every $\abs{x}\geq 2\abs{z}$. Hence
\begin{equation*}
\int_{\abs{x}\geq 2\abs{z}}F(x,z) dx\leq C\abs{z}\int_{2\abs{z}}^\infty\frac{dx}{x^2}=\widetilde{C}
\end{equation*}
which does not depend on $\alpha=(\sigma,y)\in A$. So that the theorem of Z\'{o} applies and $L^*$ is weak type (1,1). Standard arguments of interpolation show that $L^*$ is bounded on $L^p(\mathbb{R})$ for $1<p\leq\infty$.
\end{proof}
\begin{theorem}\label{thm:conditionforConcentration}
Assume, as in Theorem~\ref{thm:applyingZo} that $n=1$. Let $\Sigma: \mathbb{R}^+\to (0,2)$ be given. Assume that the function $\Sigma$ satisfies the condition $y^{\Sigma(y)+1/2}(\Sigma(y))^{-1}\to 0$ when $y\to 0^+$. Then the kernels $\{v(x,y;\Sigma(y)):y>0\}$ concentrate around the origin. Hence $(v(\cdot,y;\Sigma(y))\ast f)(x)\to f(x)$ for almost every $x\in \mathbb{R}$ and every $f\in L^p(\mathbb{R})$, $1\leq p\leq\infty$.
\end{theorem}
\begin{proof}
Since $\mathcal{J}_{-\tfrac{1}{2}}=\sqrt{\frac{2}{\pi\rho}}\cos\rho$ (see \cite{BoChan49}), then integrating by parts
\begin{align*}
v\left(x,y;\Sigma(y)\right) &=y^{-1}v\left(\tfrac{x}{y},1;\Sigma(y)\right)=\frac{1}{y}\sqrt{2\pi}\sqrt{\abs{x}}\int_0^\infty e^{-t^{\Sigma(y)}}t^{1/2}\mathcal{J}_{-\tfrac{1}{2}}\left(\tfrac{\abs{x}}{y}t\right) dt\\
&=\frac{2}{\sqrt{y}}\int_0^\infty e^{-t^{\Sigma(y)}}\cos\left(\tfrac{\abs{x}}{y}t\right) dt\\
&=\frac{2}{\sqrt{y}}\left[e^{-t^{\Sigma(y)}}\frac{y}{\abs{x}}\left.\sin\left(\tfrac{\abs{x}}{y}t\right)\right|^\infty_0
+\Sigma(y)\int_0^\infty e^{-t^{\Sigma(y)}}t^{\Sigma(y)-1}\frac{y}{\abs{x}}\sin\left(\tfrac{\abs{x}}{y}t\right)\right]\\
&=\frac{2}{\sqrt{y}}\Sigma(y)\frac{\sqrt{y}}{\abs{x}}\int_0^\infty e^{-t^{\Sigma(y)}}t^{\Sigma(y)}\sin\left(\tfrac{\abs{x}}{y}t\right)\frac{dt}{t}\\
&=2\Sigma(y)\frac{\sqrt{y}}{\abs{x}}\left(\frac{\abs{x}}{y}\right)^{1-\Sigma(y)}
\int_0^\infty e^{-t^{\Sigma(y)}}t^{\Sigma(y)}\left(\frac{\abs{x}}{y}t\right)^{\Sigma(y)-1}
\sin\left(\tfrac{\abs{x}}{y}t\right)dt\\
&=\frac{2\Sigma(y)}{\abs{x}^{\Sigma(y)}}y^{\Sigma(y)}y^{-1/2}\int_0^\infty e^{-t^{\Sigma(y)}}\xi\left(\frac{\abs{x}t}{y}\right) dt
\end{align*}
where $\xi(s)=s^{\Sigma(y)-1}\sin s$. Define $\Theta(t)=\int_0^t\xi(s)ds$, ($\Theta(0)=0$ and $\Theta'(t)=\xi(t)$). Hence
\begin{align*}
v\left(x,y;\Sigma(y)\right) &=
\frac{2\Sigma(y)}{\abs{x}^{\Sigma(y)}}y^{\Sigma(y)}y^{-1/2}\left[\left.e^{-t^{\Sigma(y)}}\frac{y}{\abs{x}}
\Theta\left(\frac{\abs{x}t}{y}\right)\right|^\infty_0 -\int_0^\infty(-\Sigma(y))t^{\Sigma(y)-1}e^{-t^{\Sigma(y)}}\frac{y}{\abs{x}}\Theta\left(\frac{\abs{x}t}{y}\right)dt\right]\\
&=\frac{2\Sigma(y)^2y^{\Sigma(y)}}{\abs{x}^{1+\Sigma(y)}}\sqrt{y}\int_0^\infty e^{-t^{\Sigma(y)}}t^{\Sigma(y)}\Theta\left(\frac{\abs{x}t}{y}\right)\frac{dt}{t}.
\end{align*}
When $\frac{\abs{x}t}{y}\leq\pi$, we have
$\abs{\Theta\left(\frac{\abs{x}t}{y}\right)}\leq\int_0^\pi s^{\Sigma(y)-1}ds=\frac{\pi^{\Sigma(y)}}{\Sigma(y)}.$
If, on the other hand, $\frac{\abs{x}t}{y}>\pi$, we have an alternating series with decreasing absolute values, so that
\begin{align*}
\int_0^{\frac{\abs{x}t}{y}}s^{\Sigma(y)-1}\sin s ds &= \sum_{j=1}^{\infty}(-1)^{j+1}\int_{(j-1)\pi}^{j\pi}s^{\Sigma(y)-1}\mathcal{X}_{\left[0,\tfrac{\abs{x}t}{y}\right]}(s)\abs{\sin s} ds\\
&\leq\int_0^\pi s^{\Sigma(y)-1}\sin s\, ds\leq\int_0^\pi s^{\Sigma(y)-1} ds=\frac{\pi^{\Sigma(y)}}{\Sigma(y)} .
\end{align*}
Hence, in any case
\begin{equation*}
\abs{\Theta\left(\frac{\abs{x}t}{y}\right)}\leq\frac{\pi^{\Sigma(y)}}{\Sigma(y)}.
\end{equation*}
This estimate provides an estimate for $v(x,y;\Sigma(y))$. In fact
\begin{align*}
v(x,y;\Sigma(y))&\leq\frac{2\Sigma(y)\pi^{\Sigma(y)}\sqrt{y}y^{\Sigma(y)}}{\abs{x}^{1+\Sigma(y)}}
\int_0^\infty e^{-t^{\Sigma(y)}}t^{\Sigma(y)}\frac{dt}{t}\\
&\leq\frac{2\Sigma(y)\pi^{2}\sqrt{y}y^{\Sigma(y)}}{\abs{x}^{1+\Sigma(y)}}.
\end{align*}
In order to check the concentration property for the family $\{v(x,y;\Sigma(y)): y>0\}$, take a positive $\lambda$ which we may assume less than one, and let us estimate the integrals outside the interval of length $2\lambda$ centered at the origin,
\begin{align*}
\int_{\abs{x}\geq\lambda}v(x,y;\Sigma(y))dx &\leq 2\pi^2\sqrt{y}y^{\Sigma(y)} 2\int_{\lambda}^{\infty}\frac{dx}{x^{1+\Sigma(y)}}\\
&=4\pi^2\sqrt{y}y^{\Sigma(y)}\left(-\frac{1}{\Sigma(y)}\left.x^{-\Sigma(y)}\right|^{\infty}_{\lambda}\right)\\
&\leq\frac{4\pi^2}{\lambda^2}\frac{\sqrt{y}y^{\Sigma(y)}}{\Sigma(y)}.
\end{align*}
From the hypothesis on $\Sigma$ the last expression tends to zero when $y\to 0^+$.
\end{proof}

Notice that small powers of $y$ satisfy the condition for $\Sigma$. In fact, take $0<\varepsilon<1/2$. Then with $\Sigma(y)=y^\varepsilon$ we get that $\tfrac{y^{(1/2+y^{\varepsilon})}}{y^{\varepsilon}}=y^{1/2-\varepsilon}y^{y^{\varepsilon}}$ which tends to zero as $y\to 0^+$. The condition on $\Sigma$ in Theorem~\ref{thm:conditionforConcentration} is strictly less restrictive than that in Theorem~\ref{thm:maximaloperatorforP}. Let us finally observe that the structural simplicity of the functions $\mathcal{J}_{1/2}$ and $\mathcal{J}_{-1/2}$ allowed us to get the needed estimates by integration by parts. When dimension $n$ of the underlying space is larger than one the Bessel function involved in \eqref{eq:kernelBluGetgeneral} is less simple. So that our methods in not straightforwardly extendable to dimension larger than one.

\section{Stable Markov kernels and Harnack's type inequality in $\mathbb{R}^n$}\label{sec:StableHarnackEuclidean}
Let us start by defining the concepts of stability, Harnack, concentration and approximation of the identity taking as illustration the Cauchy-Poisson biparametric family $\mathcal{P}=\{P^\sigma_y: 0<\sigma<2, y>0\}$. We shall extend these concepts and the results relating them for more general non-convolution kernels.
A symmetric \textbf{Markov kernel} in $\mathbb{R}^n$ is a nonnegative measurable and symmetric function $K$ defined in $\mathbb{R}^n\times\mathbb{R}^n$ such that $\int_{\mathbb{R}^n}K(x,z)dz=1$ for every $x\in \mathbb{R}^n$.

The next result gives a proof of the Harnack inequality, stated in the introduction, for the family $\mathcal{P}$.
\begin{proposition}\label{propo:supinfPoisson}
Let $0<\gamma<1$ be given. For $x\in \mathbb{R}^n\setminus\{0\}$ set the ball $B(x,\gamma\abs{x})$, $\mathfrak{s}_\gamma(x,y,\sigma)=\sup_{z\in B(x,\gamma\abs{x})}P^\sigma_y(z)$ and $\mathfrak{i}_\gamma(x,y,\sigma)=\inf_{z\in B(x,\gamma\abs{x})}P^\sigma_y(z)$. Then
\begin{equation*}
\frac{\mathfrak{s}_\gamma(x,y,\sigma)}{\mathfrak{i}_\gamma(x,y,\sigma)}<\left(\frac{1+\gamma}{1-\gamma}\right)^{n+\sigma}
\end{equation*}
uniformly in $x\in \mathbb{R}^n\setminus\{0\}$, $y>0$. Moreover, $\sup_{x,y}\frac{\mathfrak{s}_\gamma}{\mathfrak{i}_\gamma}=\left(\frac{1+\gamma}{1-\gamma}\right)^{n+\sigma}$.
\end{proposition}
\begin{proof}
Recall that $P^\sigma_y(z)=y^{-n}P^\sigma_1(y^{-1}z)=(I(\sigma))^{-1}y^\sigma(y^2+\abs{z}^2)^{-(n+\sigma)/2}$. Hence, for $x\neq 0$ and $z\in B(x,\gamma\abs{x})$ we have
\begin{equation*}
\mathfrak{s}_\gamma(x,y,\sigma)=\frac{y^\sigma}{I(\sigma) (y^2+(1-\gamma)^2\abs{x}^2)^{\tfrac{n+\sigma}{2}}},
\end{equation*}
and
\begin{equation*}
\mathfrak{i}_\gamma(x,y,\sigma)=\frac{y^\sigma}{I(\sigma) (y^2+(1+\gamma)^2\abs{x}^2)^{\tfrac{n+\sigma}{2}}}.
\end{equation*}
So that
\begin{equation*}
\frac{\mathfrak{s}_\gamma(x,y,\sigma)}{\mathfrak{i}_\gamma(x,y,\sigma)}=
\left[\frac{1+(1+\gamma)^2\left(\frac{\abs{x}}{y}\right)^2}{1+(1-\gamma)^2\left(\frac{\abs{x}}{y}\right)^2}\right]^{\tfrac{n+\sigma}{2}}.
\end{equation*}
Since the expression inside the rectangular brackets is strictly increasing as a function of $\left(\tfrac{\abs{x}}{y}\right)^2$, the supremum of this expression is its limit for $\tfrac{\abs{x}}{y}\to\infty$. So that $\tfrac{\mathfrak{s}_\gamma}{\mathfrak{i}_\gamma}<\left(\tfrac{1+\gamma}{1-\gamma}\right)^{n+\sigma}$ and $\sup_{x,y}\frac{\mathfrak{s}_\gamma}{\mathfrak{i}_\gamma}=\left(\tfrac{1+\gamma}{1-\gamma}\right)^{n+\sigma}$.
\end{proof}
Notice that, since $0<\sigma<2$, the number $\left(\tfrac{1+\gamma}{1-\gamma}\right)^{n+2}$ is an upper bound for $\frac{\mathfrak{s}_\gamma}{\mathfrak{i}_\gamma}$ which is also uniform in $\sigma$.

Given a Markov kernel $K$ defined in $\mathbb{R}^n\times \mathbb{R}^n$ we shall say that $K$ satisfies a \textbf{Harnack condition with constants \boldmath{$0<\gamma<1$} and \boldmath{$H\geq 1$}}, and write {\boldmath{$K\in\mathcal{H}(\gamma,H)$}} if the inequality
\begin{equation*}
\sup_{z\in B(\xi,\gamma\abs{x-\xi})}K(x,z)\leq H\inf_{z\in B(\xi,\gamma\abs{x-\xi})}K(x,z)
\end{equation*}
holds for every $x\in \mathbb{R}^n$ and every $\xi\in \mathbb{R}^n$ with $\xi\neq x$. With this notation Proposition~\ref{propo:supinfPoisson} reads $P^\sigma_y\in\mathcal{H}\bigl(\gamma,\bigl(\tfrac{1+\gamma}{1-\gamma}\bigr)^{n+\sigma}\bigr)$ for every $0<\gamma<1$, every $0<\sigma<2$ and every $y>0$. Notice also that $\mathcal{P}\subset\mathcal{H}\bigl(\gamma,\bigl(\tfrac{1+\gamma}{1-\gamma}\bigr)^{n+2}\bigr)$.
\begin{proposition}
For $0<\sigma<2$ and $y>0$, we have that $\abs{x-z}^{n+\sigma}P^\sigma_y(x-z)\to y^\sigma (I(\sigma))^{-1}$ when $\abs{x-z}\to\infty$.
\end{proposition}
\begin{proof}
Write
\begin{equation*}
\abs{x-z}^{n+\sigma}\frac{y^\sigma}{I(\sigma)\bigl(y^2+\abs{x-z}^2\bigr)^{\tfrac{n+\sigma}{2}}}
=\frac{y^\sigma}{I(\sigma)}\frac{\abs{x-z}^{n+\sigma}}{(y^2+\abs{x-z}^2)^{\tfrac{n+\sigma}{2}}}
=\frac{y^\sigma}{I(\sigma)}\frac{1}{\biggl(\frac{y^2}{\abs{x-z}^2}+1\biggr)^{\tfrac{n+\sigma}{2}}}
\end{equation*}
and take limit for $\abs{x-z}\to\infty$.
\end{proof}

Given a Markov kernel $K$ defined in $\mathbb{R}^n\times \mathbb{R}^n$ we say that \textbf{\boldmath{$K$} is \boldmath{$\sigma$}-stable with parameter \boldmath{$\alpha>0$}} and
we write   {\boldmath{$K\in\mathcal{S}(\sigma,\alpha)$}} if
\begin{equation*}
\lim_{\abs{x-z}\to\infty}\abs{x-z}^{n+\sigma}K(x,z)=\alpha.
\end{equation*}
With this notation $P^\sigma_y\in\mathcal{S}(\sigma,y^{\sigma}(I(\sigma))^{-1})$ for every $0<\sigma<2$ and every $y>0$.
In \S~2 we proved that for the families $\mathcal{P}$ and $\mathcal{L}$ we have some non-standard instances of concentration about the origin. That is the case when not only the mollification parameter tends to zero, but also $\sigma$ may change. For the non-convolution case we may ask for the concentration of a family of Markov kernels about the diagonal. A family of Markov kernels indexed by a parameter $\tau\in (0,1)$, $\mathcal{K}=\{K_\tau(x,z), 0<\tau<1\}$, is said to \textbf{concentrate} or to satisfy the \textbf{concentration property} ({\boldmath{$\mathcal{K}\in\mathcal{C}$}}) if $\int_{\abs{x-z}\geq\lambda}K_\tau(x,z)dz$ tends to zero when $\tau\to 0$ for every $\lambda>0$, uniformly in $x\in \mathbb{R}^n$. In Section~2 we proved that the families $\{P^{\Sigma(y)}_y(x-z): 0<y<1\}$ when $y^{\Sigma(y)}\to 0^+$ when $y\to 0^+$ and $\{L^{\Sigma(y)}_y(x-z): 0<y<1\}$ when $y^{1/2+\Sigma(y)}(\Sigma(y))^{-1}\to 0^+$ when $y\to 0^+$ belong to $\mathcal{C}$.

Let  us state and prove a result showing that Harnack and stability implies concentration and boundedness of the maximal operator for adequate families of Markov kernels.
\begin{theorem}\label{thm:suffconditionHarnackStability}
Let $\sigma>0$ be given. Assume that for each $0<\alpha<1$ a Markov kernel $K^\sigma_\alpha\in\mathcal{S}(\sigma,\alpha)$ is given. If for some $0<\gamma<1$ we have that $\mathcal{K}=\{K^\sigma_\alpha: 0<\alpha<1\}\subset\mathcal{H}\left(\gamma,\left(\tfrac{1+\gamma}{1-\gamma}\right)^{n+\sigma}\right)$, then
\begin{enumerate}[(a)]
\item $\mathcal{K}\in\mathcal{C}$;
\item $\mathcal{K}^*f(x)=\sup_{0<\alpha<1}\abs{\int_{\mathbb{R}^n}K^\sigma_\alpha(x,z)f(z)dz}$ is of weak type (1,1) and bounded in $L^p(\mathbb{R}^n)$, $1<p\leq\infty$;
\item $\lim_{\alpha\to 0}\int_{\mathbb{R}^n}K^\sigma_\alpha(x,z)f(z)dz=f(x)$ almost everywhere for $f\in L^p(\mathbb{R}^n)$ with $1\leq p\leq\infty$.
\end{enumerate}
\end{theorem}
\begin{proof}
[Proof of (a)]
Fix $x\in \mathbb{R}^n$ and $0<\lambda<1$. Since each $K^\sigma_\alpha$ belongs to $\mathcal{S}(\sigma,\alpha)$, then there exists $R(\alpha)$ large enough such that $\abs{x-z}^{n+\sigma}K^\sigma_\alpha(x,z)<2\alpha$ whenever $\abs{x-z}\geq R(\alpha)$. In other words $K^\sigma_\alpha(x,z)\leq\tfrac{2\alpha}{\abs{x-z}^{n+\sigma}}$ for $\abs{x-z}\geq R(\alpha)$. Now we proceed to use that $\mathcal{K}\subset\mathcal{H}\left(\gamma,\left(\tfrac{1+\gamma}{1-\gamma}\right)^{n+\sigma}\right)$, for some $0<\gamma<1$, in order to estimate $K^\sigma_\alpha(x,z)$ inside the ball $B(x,R(\alpha))$. The annuli
\begin{equation*}
A_j=\set{z\in \mathbb{R}^n: \left(\frac{1-\gamma}{1+\gamma}\right)^{j+1}R(\alpha)\leq
\abs{x-z}<\left(\frac{1-\gamma}{1+\gamma}\right)^{j}R(\alpha)},\,\, j=0,1,2,\ldots
\end{equation*}
 provide a covering of the set $B(x,R(\alpha))\setminus \{x\}$. Hence, with $J=J(\lambda,R(\alpha))$ the first integer for which $\left(\frac{1-\gamma}{1+\gamma}\right)^{j+1}R(\alpha)<\lambda$, we get that
$\set{z: \abs{x-z}\geq \lambda}\subset \bigcup_{j=0}^{J}A_j \cup\set{z:\abs{x-z}\geq R(\alpha)}$. For $z$ such that $\abs{x-z}\geq R(\alpha)$ we have that $K^\sigma_\alpha(x,z)\leq\tfrac{2\alpha}{\abs{x-z}^{n+\sigma}}$. Let us now get bounds for $K^\sigma_\alpha(x,z)$ with $z$ in the annuli $A_j$. Let us start by $j=0$. Notice that if
$\tfrac{1-\gamma}{1+\gamma}R(\alpha)<\abs{x-z}<R(\alpha)$ then $z$ also belongs to a ball $\widetilde{B}=B(\xi,\gamma\abs{x-\xi})$ with $\widetilde{B}\cap\{z:\abs{x-z}\geq R(\alpha)\}\neq\emptyset$. Hence, with $\theta\in\widetilde{B}\cap\{z:\abs{x-z}\geq R(\alpha)\}$, we have
\begin{align*}
K^\sigma_\alpha(x,z) &\leq \sup_{\eta\in \widetilde{B}}K^\sigma_\alpha(x,\eta)\\
&\leq\left(\frac{1+\gamma}{1-\gamma}\right)^{n+\sigma} \inf_{\eta\in \widetilde{B}}K^\sigma_\alpha(x,\eta)\\
&\leq \left(\frac{1+\gamma}{1-\gamma}\right)^{n+\sigma}K^\sigma_\alpha(x,\theta)\\
&\leq \left(\frac{1+\gamma}{1-\gamma}\right)^{n+\sigma}\frac{2\alpha}{R(\alpha)^{n+\sigma}}.
\end{align*}
By iteration, we have that, for $z\in A_j$ the upper estimate
\begin{equation*}
K^\sigma_\alpha(x,z)\leq \frac{2\alpha}{R(\alpha)^{n+\sigma}}\left(\frac{1+\gamma}{1-\gamma}\right)^{(j+1)(n+\sigma)}.
\end{equation*}
Hence, for $\abs{x-z}\geq\lambda$ we have that
\begin{equation*}
K^\sigma_\alpha(x,z) \leq \frac{2\alpha}{R(\alpha)^{n+\sigma}}
\sum_{j=0}^{J}\left(\frac{1+\gamma}{1-\gamma}\right)^{(j+1)(n+\sigma)}\mathcal{X}_{A_j}(z) + \frac{2\alpha}{\abs{x-z}^{n+\sigma}}\mathcal{X}_{\{\abs{x-z}\geq R(\alpha)\}}(z).
\end{equation*}
Hence
\begin{align*}
\int_{\set{z:\abs{x-z}\geq\lambda}}K^\sigma_\alpha(x,z) dz & \leq 2\alpha\left\{\frac{1}{R(\alpha)^{n+\sigma}}\sum_{j=0}^{J}\left(\frac{1+\gamma}{1-\gamma}\right)^{(j+1)(n+\sigma)}\abs{A_j} +\omega_n\int_{R(\alpha)}^{\infty}\rho^{n-1}\frac{d\rho}{\rho^{n+\sigma}}\right\}\\
&\leq 2\alpha\omega_n\left\{\frac{1}{R(\alpha)^{n+\sigma}}\sum_{j=0}^{J}\left(\frac{1+\gamma}{1-\gamma}\right)^{(j+1)(n+\sigma)}
\left(\frac{1-\gamma}{1+\gamma}\right)^{jn}R(\alpha)^n + \frac{1}{\sigma R(\alpha)^\sigma}\right\}\\
&=\frac{2\omega_n}{R(\alpha)^\sigma}\alpha\left\{\frac{1}{\sigma}+\sum_{j=0}^{J}\left(\frac{1+\gamma}{1-\gamma}\right)^{n+\sigma(j+1)}\right\}\\
&=\frac{2\omega_n}{R(\alpha)^\sigma}\alpha\left\{\frac{1}{\sigma}+\left(\frac{1+\gamma}{1-\gamma}\right)^{n+\sigma}
\frac{1}{\left(\tfrac{1+\gamma}{1-\gamma}\right)^{\sigma}-1}
\left[\left(\frac{1+\gamma}{1-\gamma}\right)^{(J+1)\sigma}-1\right]\right\}.
\end{align*}
Now from the choice of $J$ as the smallest integer for which $\left(\frac{1-\gamma}{1+\gamma}\right)^{j+1}R(\alpha)<\lambda$, we see that
\begin{equation*}
\left(\frac{1+\gamma}{1-\gamma}\right)^{(J+1)\sigma}>\left(\frac{R(\alpha)}{\lambda}\right)^\sigma
\geq\left(\frac{1+\gamma}{1-\gamma}\right)^{J\sigma}.
\end{equation*}
Hence
\begin{align*}
\int_{\set{z:\abs{x-z}\geq\lambda}}K^\sigma_\alpha(x,z) dz & \leq
\frac{2\omega_n}{R(\alpha)^\sigma}\alpha\left\{\frac{1}{\sigma}+\left(\frac{1+\gamma}{1-\gamma}\right)^{n+\sigma}\frac{1}{\left(\tfrac{1+\gamma}{1-\gamma}\right)^{\sigma}-1}
\left(\frac{R(\alpha)}{\lambda}\right)^\sigma\right\}\\
&\leq\frac{2\omega_n\alpha}{\sigma}+2\omega_n\alpha\left(\frac{1+\gamma}{1-\gamma}\right)^{n}
\frac{(1+\gamma)^\sigma}{(1+\gamma)^\sigma-(1-\gamma)^\sigma}\frac{1}{\lambda^\sigma}.
\end{align*}
The last expression tends to zero as when $\alpha\to 0$ and the concentration is proved.

\medskip
\noindent\textit{Proof of (b).}
Let us show that $\mathcal{K}^*f$ is bounded above by the Hardy-Littlewood maximal function $Mf$ pointwise. Let us start by building a regularization $\widetilde{K}^\sigma_\alpha$ of $K^\sigma_\alpha$. Define
\begin{equation*}
\widetilde{K}^\sigma_\alpha(x,z)=\fint_{\zeta\in B(z,\gamma\abs{z-x})}K^\sigma_\alpha(x,\zeta)d\zeta=
\frac{1}{\abs{B(z,\gamma\abs{z-x})}}\int_{\zeta\in B(z,\gamma\abs{z-x})}K^\sigma_\alpha(x,\zeta)d\zeta.
\end{equation*}
Notice that from the fact that $\mathcal{K}\subset\mathcal{H}\bigl(\gamma,\bigl(\tfrac{1+\gamma}{1-\gamma}\bigr)^{n+\sigma}\bigr)$ we have that $\widetilde{K}^\sigma_\alpha$ provides an upper estimate for $K^\sigma_\alpha$. In fact
\begin{align*}
K^\sigma_\alpha(x,z)&\leq \sup_{\zeta\in B(z,\gamma\abs{z-x})}K^\sigma_\alpha(x,\zeta)\\
&\leq \left(\frac{1+\gamma}{1-\gamma}\right)^{n+\sigma}\inf_{\zeta\in B(z,\gamma\abs{z-x})}K^\sigma_\alpha(x,\zeta)\\
&\leq \left(\frac{1+\gamma}{1-\gamma}\right)^{n+\sigma}\fint_{B(z,\gamma\abs{z-x})}K^\sigma_\alpha(x,\zeta)d\zeta\\
&= \left(\frac{1+\gamma}{1-\gamma}\right)^{n+\sigma}\widetilde{K}^\sigma_\alpha(x,z).
\end{align*}
Actually also $\widetilde{K}^\sigma_\alpha$ is bounded above by $K^\sigma_\alpha$ and they are in fact pointwise equivalent. Hence
\begin{align*}
\abs{\int_{\mathbb{R}^n}K^\sigma_\alpha(x,z)f(z)dz}&\leq\int_{\mathbb{R}^n}K^\sigma_\alpha(x,z)\abs{f(z)}dz\\
&\leq \left(\frac{1+\gamma}{1-\gamma}\right)^{n+\sigma}\int_{\mathbb{R}^n}\widetilde{K}^\sigma_\alpha(x,z)\abs{f(z)}dz\\
&=\left(\frac{1+\gamma}{1-\gamma}\right)^{n+\sigma}\frac{\omega_n}{\gamma^n}\int_{\mathbb{R}^n}\left(\frac{1}{\abs{z-x}^n}\int_{\zeta\in B(z,\gamma\abs{z-x})}K^\sigma_\alpha(x,\zeta)d\zeta\right)\abs{f(z)}dz\\
&=\frac{\omega_n}{\gamma^n}\left(\frac{1+\gamma}{1-\gamma}\right)^{n+\sigma}\int_{\zeta\in \mathbb{R}^n}K^\sigma_\alpha(x,\zeta)
\left(\int_{\{z\in\mathbb{R}^n: \abs{\zeta-z}<\gamma\abs{x-z}\}}\frac{\abs{f(z)}}{\abs{x-z}^n}dz\right) d\zeta.
\end{align*}
Let us estimate the inner integral. Two geometric remarks are in order.

\textit{Claim~1}. For  $\zeta\neq x$ the set $\{z\in\mathbb{R}^n: \abs{\zeta-z}<\gamma\abs{x-z}\}\subset B\left(x,\tfrac{\abs{\zeta-x}}{1-\gamma}\right)$.
In fact, for $z$ such that $\abs{\zeta-z}<\gamma\abs{x-z}$ we have that $\abs{z-x}\leq \abs{z-\zeta}+\abs{\zeta-x}<\gamma\abs{x-z}+\abs{\zeta-x}$. Hence $\abs{z-x}(1-\gamma)<\abs{\zeta-x}$ and $z\in B\left(x,\tfrac{\abs{\zeta-x}}{1-\gamma}\right)$.

 \textit{Claim~2}. For $z$ such that $\abs{\zeta-z}<\gamma\abs{x-z}$, we have $\abs{x-\zeta}\leq (1+\gamma)\abs{z-x}$. In fact, $\abs{x-\zeta}\leq\abs{x-z}+\abs{z-\zeta}<(1+\gamma)\abs{x-z}$.

\noindent From Claim~2 we have that
\begin{equation*}
\int_{\{z\in\mathbb{R}^n: \abs{\zeta-z}<\gamma\abs{x-z}\}}\frac{\abs{f(z)}}{\abs{x-z}^n}dz\leq \frac{(1+\gamma)^n}{\abs{x-\zeta}^n}\int_{\{z\in\mathbb{R}^n: \abs{\zeta-z}<\gamma\abs{x-z}\}}\abs{f(z)}dz.
\end{equation*}
On the other hand, from Claim~1 the last term is bounded by $\tfrac{(1+\gamma)^n}{\abs{x-\zeta}^n}\int_{B\bigl(x,\tfrac{\abs{\zeta-x}}{1-\gamma}\bigr)}\abs{f(z)}dz$ which is certainly bounded by a constant times the Hardy--Littlewood maximal function of $f$ at $x$. Precisely
\begin{equation*}
\int_{\{z\in\mathbb{R}^n: \abs{\zeta-z}<\gamma\abs{x-z}\}}\frac{\abs{f(z)}}{\abs{x-z}^n}dz\leq \omega_n\left(\frac{1+\gamma}{1-\gamma}\right)^{n}Mf(x).
\end{equation*}
Hence
\begin{equation}\label{eq:estimateforMHL}
\abs{\int_{\mathbb{R}^n}K^\sigma_\alpha(x,z)f(z)dz}\leq\frac{\omega_n^2}{\gamma^n}
\left(\frac{1+\gamma}{1-\gamma}\right)^{2n+\sigma}Mf(x)\int_{\zeta\in \mathbb{R}^n}K^\sigma_\alpha(x,\zeta)d\zeta=C(\gamma,\sigma,n)Mf(x).
\end{equation}
Since this upper bound is uniform in $\alpha$ we get the results of \textit{(b)}.

\medskip
\noindent\textit{Proof of (c).}
For a continuous function $f$ with compact support, we have that
\begin{align*}
\abs{\int_{\mathbb{R}^n}K^\sigma_\alpha(x,z)f(z)dz-f(x)}&\leq\int_{\mathbb{R}^n}K^\sigma_\alpha(x,z)\abs{f(z)-f(x)}dz\\
&\leq 2\norm{f}_\infty\int_{\{z:\abs{x-z}\geq\lambda\}}K^\sigma_\alpha(x,z)dz + \varepsilon\int_{\{z:\abs{x-z}<\lambda\}}K^\sigma_\alpha(x,z)dz,
\end{align*}
with $\lambda$ such that $\abs{f(x)-f(z)}<\varepsilon$. From \textit{(a)} the concentration property of $\mathcal{K}$ proves the result for continuous $f$. Now the argument is the usual by density and type of the operator $\mathcal{K}^*$.
\end{proof}

Let us remark that the precise values of the constants in the Harnack condition is only used in proving the concentration property in part \textit{(a)} of Theorem~\ref{thm:suffconditionHarnackStability}. For the boundedness properties of $\mathcal{K}^*$ those values are irrelevant. From inequality \eqref{eq:estimateforMHL} in the proof of \textit{(b)} in Theorem~\ref{thm:suffconditionHarnackStability} we get weak type and boundedness for a larger maximal function. Let $\sigma_0>0$ fixed. Define
\begin{equation*}
\mathcal{K}^{**}f(x)=\sup_{\substack{0<\alpha<1\\0<\sigma\leq\sigma_0}}\abs{\int K^\sigma_\alpha(x,z)f(z)dz}.
\end{equation*}
\begin{corollary}
The operator $\mathcal{K}^{**}$ is of weak type (1,1) and bounded on $L^p$ for $1<p\leq\infty$.
\end{corollary}
\begin{proof}
We only have to observe that $C(\gamma,\sigma,n)$ in \eqref{eq:estimateforMHL} is bounded above by $C(\gamma,\sigma_0,n)$.
\end{proof}

\section{The general setting}\label{sec:generalsetting}
One of the most significant aspects of the concepts involved in the previous sections is that they can completely be given
in terms of metric and measure. A more subtle analysis of the proofs in Section~3 above unveils the fact that the doubling condition for the measure
of balls becomes a central issue. Hence our general setting shall be that of spaces of homogeneous type. The basic reference for the fundamental
concepts and results is \cite{MaSe79Lip}. Approximate identities and an extension of the result of Z\'{o} have been considered in the abstract setting in \cite{Aimar85}.

Let $X$ be a set. A function $d$ defined on $X\times X$ is said to be a quasi-distance if it is nonnegative, symmetric, vanishes only on the diagonal and for some constant $\tau\geq 1$  the extended triangle inequality $d(x,z)\leq \tau[d(x,y)+d(y,z)]$ holds for every $x$, $y$ and $z$ in $X$. Let $\mathcal{F}$ be a $\sigma$-algebra of subsets of $X$ that contains the $d$-balls $B=B(x,r)=\set{y\in X: d(x,y)<r}$, $x\in X$, $r>0$. A positive measure $\mu$ defined on $\mathcal{F}$ is said to be doubling with respect to $(X,d)$ if there exists $A\geq 1$ such that $0<\mu(B(x,2r))\leq A\mu(B(x,r))<\infty$ for every $x\in X$ and every $r>0$. When $\mu$ is doubling on $(X,d)$, $(X,d,\mu)$ is said to be a space of homogeneous type. Actually the homogeneity alluded in the name of the structure is quite weak and not only Euclidean spaces, self similar fractals or parabolic metrics provide spaces of homogeneous type. For example $(\mathbb{R},\abs{x-y},(\sqrt{x})^{-1}dx)$ and $(\Omega,\abs{x-y},dx)$ with $\Omega$ a Lipschitz domain in $\mathbb{R}^n$ are spaces of homogeneous type. The numbers $\tau$ and $A$ are called the geometric constants of $(X,d,\mu)$.

Let $(X,d,\mu)$ be a space of homogeneous type. A nonnegative, symmetric, and measurable function $K$ defined on $X\times X$ is said to be a symmetric Markov kernel in $X$ if $\int_X K(x,y)d\mu(y)=1$ for every $x\in X$.
Observe that the concept the Markov kernel depends only on the measure structure of $X$ but not explicitly on the metric. On the other hand, Harnack type inequality depends only on the metric structure. At a first glance the stability condition seems to be a metric concept. But since in the expression $\abs{x-y}^{n+\sigma}$ the parameter $n$ is the dimension of the underlying space, its right interpretation is $\abs{B(x,\abs{x-y})}^{1+\sigma/n}$. Hence stability should be interpreted as a mixed concept that involves metric and measure.

In what follows $(X,d,\mu)$ denotes an unbounded space of homogeneous type. Recall \cite{Aimarbook} that $(X,d)$ is unbounded if and only if $\mu(X)=+\infty$.

We say that a symmetric Markov kernel $K$ defined in $X$ satisfies a \textbf{Harnack inequality with constants \boldmath{$0<\gamma<1$} and \boldmath{$H\geq 1$}} or that \textbf{\boldmath{$K\in\mathcal{H}(\gamma,H)$}} if the inequality
\begin{equation*}
\sup_{\eta\in B(y,\gamma d(x,y))}K(x,\eta)\leq H\inf_{\eta\in B(y,\gamma d(x,y))}K(x,\eta)
\end{equation*}
holds for every $x$ and $y$ in $X$ with $y\neq x$.

The first result which follows from a uniform Harnack inequality  is the boundedness of the maximal operator induced by a family of kernels. From the well known covering lemmas of Wiener type, the Hardy--Littlewood operator
\begin{equation*}
Mf(x)=\sup_{x\in B}\frac{1}{\mu(B)}\int_B\abs{f(y)}d\mu(y)
\end{equation*}
is of weak type (1,1) and bounded in $L^p(X,\mu)$ for $1<p\leq\infty$. In the definition of $M$ the supremum is taken on the family of all balls containing $x$.
\begin{theorem}\label{thm:estimatemaximalHL}
Let $(X,d,\mu)$ be a space of homogeneous type with geometric constants $\tau$ and $A$. Let $\mathcal{K}$ be a family of symmetric Markov kernels $K$ on $X$ such that there exist $0<\gamma<\tfrac{1}{\tau}$ and $H\geq 1$ with $\mathcal{K}\subset\mathcal{H}(\gamma,H)$, then there exists a constant $C$ depending on $\tau$, $A$, $\gamma$ and $H$ such that
\begin{equation*}
\mathcal{K}^*f(x)\leq C Mf(x)
\end{equation*}
for every $x\in X$ and every measurable function $f$.
\end{theorem}
In order to prove Theorem~\ref{thm:estimatemaximalHL} let us introduce a regularization $\widetilde{K}$ of each $K\in\mathcal{K}$, given by
\begin{equation*}
\widetilde{K}(x,y)=\frac{1}{\mu(B(y,\gamma d(x,y)))}\int_{z\in B(y,\gamma d(x,y))}K(x,z)d\mu(z)
\end{equation*}
for $x\neq y$ and $\widetilde{K}(x,x)=K(x,x)$.
We shall write
\begin{equation*}
\widetilde{K}^*f(x)=\sup_{K\in \mathcal{K}}\abs{\int_X \widetilde{K}(x,y)f(y)d\mu(y)}.
\end{equation*}
\begin{lemma}\label{lemma:equivalentkernelsK}
For every $K\in \mathcal{K}$ and every $(x,y)\in X\times X$ we have $K(x,y)\leq H \widetilde{K}(x,y)\leq H^2 K(x,y)$.
\end{lemma}
\begin{proof}
First notice that since $\mathcal{K}\subset\mathcal{H}(\gamma,H)$, we have that for $x\neq y$
\begin{align*}
K(x,y) &= \frac{1}{\mu(B(y,\gamma d(x,y)))}\int_{B(y,\gamma d(x,y))}K(x,y)d\mu(z)\\
&\leq \frac{1}{\mu(B(y,\gamma d(x,y)))}\int_{B(y,\gamma d(x,y))}\sup_{y\in B(y,\gamma d(x,y))} K(x,\eta)d\mu(z) \\
&\leq \frac{H}{\mu(B(y,\gamma d(x,y)))}\int_{B(y,\gamma d(x,y))}\inf_{y\in B(y,\gamma d(x,y))} K(x,\eta)d\mu(z) \\
&\leq H \widetilde{K}(x,y).
\end{align*}
The second inequality follows from a second application of the Harnack condition $\mathcal{H}(\gamma,H)$.
\end{proof}

Since in our general setting  the space may have atoms, which coincide with the isolated points, the next lemma provides a bound of $K$ on the diagonal in terms of the measures of atoms.
\begin{lemma}\label{lemma:estimateatoms}
Let $\mathcal{K}$ be as in Theorem~\ref{thm:estimatemaximalHL}, then $\sup_{K\in\mathcal{K}, x\in X}K(x,x)\mu(\{x\})\leq 1$.
\end{lemma}
\begin{proof}
We only have to observe that $K(x,x)\mu(\{x\})=\int_{z\in\{x\}}K(x,z)d\mu(z)\leq\int_{z\in X}K(x,z)d\mu(z)=1$,
for every $K\in\mathcal{K}$ and every $x\in X$.
\end{proof}

\begin{proof}[Proof of Theorem~\ref{thm:estimatemaximalHL}]
From Lemma~\ref{lemma:equivalentkernelsK} we have that $K^*f(x)\leq H {\widetilde{K}}^*f(x)$ for every $f$ and every $x$. Hence it is enough to prove that there exists a constant $C=C(\tau,A,\gamma,H)>0$ such that for every measurable and nonnegative function $f$ we have that
\begin{equation*}
\int_{y\in X}\widetilde{K}(x,y) f(y)d\mu(y)\leq CMf(x),
\end{equation*}
for every $x\in X$. Fix, then, $f\geq 0$, and $x\in X$. From Lemma~\ref{lemma:estimateatoms} and Tonelli's theorem with $E(x,z)=\{y: y\neq x, d(y,z)<\gamma d(x,y)\}$
\begin{align*}
\int_{y\in X}\widetilde{K}(x,y) & f(y)d\mu(y)\\ &= K(x,x)f(x)\mu(\{x\})+\int_{y\in X\setminus\{x\}}\biggl(\frac{1}{\mu(B(y,\gamma d(x,y)))}\int_{z\in B(y,\gamma d(x,y))}K(x,z)d\mu(z)\biggr) f(y) d\mu(y)\\
&\leq f(x) +\int_{y\in X}\int_{z\in X}\mathcal{X}_{X\setminus\{x\}}(y) K(x,z)\mathcal{X}_{B(y,\gamma d(x,y))}(z)\mu(B(y,\gamma d(x,y)))^{-1} f(y) d\mu(z) d\mu(y)\\
&=f(x)+\int_{z\in X}K(x,z)\biggl(\int_{E(x,z)}\frac{f(y)}{\mu(B(y,\gamma d(x,y)))}d\mu(y)\biggr)d\mu(z)\\
&\leq Mf(x) +\biggl(\frac{4\tau}{\gamma}\biggr)^{\log_2A}\int_{z\in X}K(x,z)\biggl(\int_{E(x,z)}\frac{f(y)}{\mu(B(y,2\tau d(x,y)))}d\mu(y)\biggr)d\mu(z).
\end{align*}
Notice that $E(x,z)\subset B(x,\tfrac{\tau}{1-\gamma\tau}d(x,z))$. In fact, for $y\in E(x,z)$ we have that $y\neq x$ and that $d(y,z)<\gamma d(x,y)$. Hence $d(x,y)\leq\tau(d(x,z)+d(z,y))<\tau d(x,z)+\gamma\tau d(x,y)$. So that $d(x,y)<\tfrac{\tau}{1-\gamma\tau}d(x,z)$. Also $B(y,2\tau d(x,y))\supset B\bigl(x,\tfrac{d(x,z)}{\tau(1+\gamma)}\bigr)$. For $\xi\in B\bigl(x,\tfrac{d(x,z)}{\tau(1+\gamma)}\bigr)$,
\begin{align*}
d(\xi,y) &\leq \tau(d(\xi,x) + d(x,y))\\
&<\tfrac{1}{1+\gamma}d(x,z)+\tau d(x,y)\\
&\leq \tfrac{\tau}{1+\gamma}d(x,y)+\tfrac{\tau}{1+\gamma}d(y,z)+\tau d(x,y)\\
&< \tau\bigl(\tfrac{1}{1+\gamma}+\tfrac{\gamma}{1+\gamma}+1\bigr)d(x,y)= 2\tau d(x,y).
\end{align*}
Hence
\begin{equation*}
\mu(B(y,2\tau d(x,y)))\geq\mu(B(x,\tfrac{1-\gamma\tau}{\tau^2(1+\gamma)}\tfrac{\tau}{1-\gamma\tau}d(x,z)))\geq \left(\tfrac{1-\gamma\tau}{\tau^2(1+\gamma)}\right)^{\log_2A}\mu(B(x,\tfrac{\tau}{1-\gamma\tau}d(x,z))).
\end{equation*}
With the above estimates for $E(x,z)$ and $\mu(B(y,2\tau d(x,y)))$, since $\int_X K(x,z)d\mu(z)=1$, we obtain
\begin{align*}
\int_{y\in X} &\widetilde{K}(x,y) f(y)d\mu(y)\\ &\leq Mf(x) + \biggl(\frac{4\tau^3(1+\gamma)}{\gamma(1-\gamma\tau)}\biggr)^{\log_2A}\int_{z\in X}K(x,z)\biggl[\frac{1}{\mu\bigl(B\bigl(x,\tfrac{\tau}{1-\gamma\tau}d(x,z)\bigr)\bigr)}
\int\limits_{B\bigl(x,\tfrac{\tau}{1-\gamma\tau}d(x,z)\bigr)}f(y) d\mu(y)\biggr] d\mu(z)\\
&\leq \biggl[1+\biggl(\frac{4\tau^3(1+\gamma)}{\gamma(1-\gamma\tau)}\biggr)^{\log_2A}\biggr] Mf(x).
\end{align*}
\end{proof}

In $\mathbb{R}^n$ the Harnack inequality for a Markov kernel $K(x,y)$ is equivalent to a Harnack type inequality on annuli. We say that $K\in\mathcal{H}_a(\gamma,H)$ if
\begin{equation*}
\sup_{y\in A(x,\gamma r,r)}K(x,y)\leq H\inf_{y\in A(x,\gamma r,r)}K(x,y)
\end{equation*}
for every $x\in X$ and every $r>0$. Here $0<\gamma<1$,  $H\geq 1$ and $A(x,\gamma r,r)=B(x,r)\setminus B(x,\gamma r)$.

Less simple is the question of the stability of a Markov kernel in the general setting. In particular the space $(X,d,\mu)$ may not be itself stable at infinity in some intuitive sense. Hence stability of a kernel becomes a property which is referred to the behavior of the space itself at infinite. Before introducing a natural setting for the definition of stability of a Markov kernel, let us review the known results and prove some new results regarding normalization of a space of homogeneous type. The first and most important result is the normalization theorem given in \cite{MaSe79Lip}. For our further application to describe stability, only the non atomic unbounded case is relevant. If $(X,d,\mu)$ is a space of homogeneous type such that the $d$-balls are open sets, $\mu(\{x\})=0$ for every $x\in X$ and $\mu(X)=+\infty$, then $\delta(x,y)=\inf\{\mu(B): x,y\in B, B \textrm{ a } d-\textrm{ball in  } X\}$ is a quasi-metric in $X$ that determines the same topology that $d$ generates on $X$ and there exist constants $0<c_1<c_2<\infty$ such that
\begin{equation*}
c_1r\leq\mu(B_\delta(x,r))\leq c_2r
\end{equation*}
for every $x\in X$ and $r>0$. Keeping track of the dependence of $c_1$, $c_2$ and the triangle constant $\widetilde{\tau}$ of $(X,\delta,\mu)$ in terms of the geometric constants $\tau$ and $A$ of $(X,d,\mu)$ we see that $c_1=1/A$, $c_2=(10\tau^2)^{\log_2A}$ and $\widetilde{\tau}=(6\tau^2)^{\log_2A}$ work.

Following the notation introduced in \cite{MaSe79Lip} we say the the space of homogeneous type $(X,\delta,\mu)$ is normal. We shall briefly write $(X,\delta,\mu)\in\mathcal{N}(\widetilde{\tau},c_1,c_2)$ when $(X,\delta,\mu)$ is an unbounded and non atomic normal space with constants $\widetilde{\tau}$, $c_1$ and $c_2$. In terms of the original structure $(X,d,\mu)$ it is worthy noticing that $\delta(x,y)\simeq\mu(B_d(x,d(x,y))\cup B_d(y,d(x,y)))\simeq\mu(B_d(x,d(x,y)))$ with constants that do not depend on $x,y\in X$. One of the advantages of $\delta$ is that $\delta$-balls are open sets and the kernels that are given as continuous functions of $\delta$, become measurable.

Normalization can be seen as a way of measuring the distance between points of the space in terms of the mass distribution in the original structure. In other words, the process of normalization improves the homogeneity of the space. The next result contains an elementary property that follows from normality that shall be useful in our extension of Theorem~\ref{thm:suffconditionHarnackStability} above.
\begin{lemma}
Let $(X,\delta,\mu)\in\mathcal{N}(\widetilde{\tau},c_1,c_2)$, then $\mu\left(B_\delta\left(x,\tfrac{(1+\varepsilon)c_2}{c_1}r\right)\setminus B_\delta(x,r)\right)\geq\varepsilon c_2 r>0$  for every $\varepsilon>0$.
\end{lemma}
\begin{proof}
$\mu\left(B_\delta\left(x,\tfrac{(1+\varepsilon)c_2}{c_1}r\right)\setminus B_\delta(x,r)\right)=\mu(B_\delta(x,\tfrac{(1+\varepsilon)c_2}{c_1}r))-\mu(B_\delta(x,r))\geq (1+\varepsilon)c_2r-c_2r=\varepsilon c_2r.$
\end{proof}
In particular, in normal spaces, annuli of some ratio are nonempty. In general, if $(X,d)$ is a quasi-metric space and $\nu>1$, we say that $A(x,r,\nu r)=B(x,\nu r)\setminus B(x,r)$ is a $\nu$-annulus in $(X,d)$ if $x\in X$ and $r>0$. We shall say that $(X,d)\in\mathcal{A}(\nu)$ if every $\nu$-annulus is nonempty. Notice that if $(X,d)\in\mathcal{A}(\nu)$ for some $\nu>1$ then the space is unbounded and no point is isolated. With this notation we have that $\mathcal{N}(\widetilde{\tau},c_1,c_2)\subseteq\mathcal{A}(\tfrac{(1+\varepsilon)c_2}{c_1})$ for every $\varepsilon>0$.
Going back to the original structure $(X,d,\mu)$ with geometric constants $\tau$ and $A$, the above result implies that $\delta$-annuli of ratio $(1+\varepsilon)A(10\tau^2)^{\log_2A}$ are nonempty for $\varepsilon>0$.

If $\nu_1<\nu_2$ then $\mathcal{A}(\nu_1)\subset\mathcal{A}(\nu_2)$. Sometimes, like in the Euclidean setting $\inf\{\nu: (X,d)\in\mathcal{A}(\nu)\}=1$. Nevertheless this is not the general situation. In fact, if we consider the subset of real numbers given by $X=\cup_{k\in \mathbb{Z}}(2k-\tfrac{1}{2},2k+\tfrac{1}{2})$ with the restriction of the Euclidean distance the index $\inf\{\nu: (X,d)\in\mathcal{A}(\nu)\}=3$. There exist spaces of homogeneous type $(X,d,\mu)$ such that $(X,d)$ does not belong to any $\mathcal{A}(\nu)$, $\nu>1$. In fact, since the subset of real numbers $X=\cup_{n\in \mathbb{N}}[2^{n}-1,2^{n}+1]$ with the restriction of the standard distance is complete, the result in \cite{Wu98} provides a doubling measure $\mu$ in $X$ with respect to the restriction $d$ of the usual distance in $\mathbb{R}$  (see also \cite{LuSa98}). Hence $(X,d,\mu)$ is a space of homogeneous type. Nevertheless $(X,d)$ does not satisfy $\mathcal{A}(\nu)$ for any $\nu>1$. Just take $x_n=2^{n}$, then $A(x_n,2,2\nu)=\{x\in X:2\leq\abs{2^{n}-x}<2\nu\}$ is empty for $n$ large enough depending on $\nu$.

Let $(X,d,\mu)$ be an unbounded and non-atomic space of homogeneous type such that the $d$-balls are open sets, with geometric constants $\tau$ and $A$.
Let $K$ be a symmetric Markov kernel defined on $X$. Let $s$ be a given positive number. We might define \textbf{\boldmath{$s$}-stability} with parameter $\alpha>0$ by the requirement
$\lim_{\delta(x,y)\to 0}\delta(x,y)^{1+s}K(x,y)=\alpha$, instead we introduce a somehow different condition substituting stability. For $R>0$ we shall say that {\boldmath{$K\in\mathcal{S}(s,\alpha, R)$}} if the inequality
\begin{equation}\label{eq:stabilityeth}
K(x,y)\leq\frac{\alpha}{\delta(x,y)^{1+s}}
\end{equation}
holds for every $x, y\in X$ such that $\delta(x,y)>R$. Which, in terms of the original structure in $X$, could be rephrased $K(x,y)\leq C\frac{\alpha}{\mu(B_d(x,d(x,y)))^{1+s}}$ for some constant $C$ and $\mu(B_d(x,d(x,y)))$ large enough.

Let us formally introduce the concept of concentration for a one parameter family of symmetric Markov kernels. Let $\mathcal{K}=\{K_\alpha(x,y): 0<\alpha<1\}$ be a family of symmetric Markov kernels defined in $(X,\delta,\mu)$ we shall say that \textbf{\boldmath{$\mathcal{K}$} concentrates} (as $\alpha$ tends to zero), and we write {\boldmath{$\mathcal{K}\in\mathcal{C}$}}, if $\int_{\delta(x,y)\geq\lambda}K_\alpha(x,y)d\mu(y)$ tends to zero for $\alpha\to 0$ for every $\lambda>0$ uniformly in $x\in X$.

We are finally in position to state and prove the main result of this section.
\begin{theorem}
Let $(X,\delta,\mu)\in\mathcal{N}(\widetilde{\tau},c_1,c_2)$. Assume that continuous functions are dense in $L^1(X,\mu)$. Let $s>0$ be given. Let $\mathcal{K}=\{K_\alpha: 0<\alpha<1\}$ be a family of symmetric Markov kernels on $X$ such that
\begin{enumerate}[(a)]
\item $\mathcal{K}\subset\mathcal{H}_a(\gamma,H)$ with $\gamma\leq(\tfrac{c_1}{2c_2})^2$;
\item there exists $R>0$ such that $K_\alpha\in\mathcal{S}(s,\alpha,R)$ for every $0<\alpha<1$.
\end{enumerate}
Then
\begin{enumerate}[(i)]
\item $\mathcal{K}\in\mathcal{C}$;
\item $\mathcal{K}^*$ is of weak type (1,1) and bounded in $L^p(X,\mu)$ for $1<p\leq\infty$;
\item $\int_XK_{\alpha}(x,y) f(y) d\mu(y)\to f(x)$ as $\alpha\to 0^+$ for every $f\in L^p(X,\mu)$, $1\leq p\leq \infty$.
\end{enumerate}
\end{theorem}
\begin{proof}
In order to prove \textit{(i)}, set $\nu=\tfrac{2c_2}{c_1}$. Since $\mathcal{N}(\widetilde{\tau},c_1,c_2)\subset\mathcal{A}(\nu)$, we see that for every $\delta$ annulus $A(x,\rho,\nu^2\rho)=A(x,\rho,\nu\rho)\cup A(x,\nu\rho,\nu^2\rho)$ there exist $x_1\in A(x,\rho,\nu\rho)$ and $x_2\in A(x,\nu\rho,\nu^2\rho)$. Hence $K_\alpha(x,x_1)\leq HK_\alpha(x,x_2)$ because of the $\mathcal{H}_a(\gamma,H)$ condition since $\gamma\leq\tfrac{1}{\nu^2}$. Moreover,
\begin{equation*}
\sup_{y\in A(x,\rho,\nu\rho)} K_\alpha(x,y)\leq H\inf_{y\in A(x,\nu\rho,\nu^2\rho)} K_\alpha(x,y).
\end{equation*}
In order to check the concentration property pick $\lambda>0$. Since $K_\alpha\in \mathcal{S}(s,\alpha)$ then, there exists $R>\lambda$ satisfying \eqref{eq:stabilityeth}. Now, for every $y\in A(x,\lambda,\nu R)$ with $\tfrac{R}{\nu^{j_0+1}}\leq\lambda<\tfrac{R}{\nu^{j_0}}$ ($j_0\sim\log_\nu\tfrac{R}{\lambda}$), we have
\begin{align*}
K_\alpha(x,y) &= \sum_{j=-1}^{j_0(R,\lambda)}K_\alpha(x,y)\mathcal{X}_{A\bigl(x,\tfrac{R}{\nu^{j+1}},\tfrac{R}{\nu^j}\bigr)}(y)\\
&\leq\frac{\alpha}{R^{1+s}}\sum_{j=-1}^{j_0(R,\lambda)}H^{j+1}\\
&= \frac{\alpha}{R^{1+s}}C(R,\lambda).
\end{align*}
Hence
\begin{align*}
\int_{\delta(x,y)\geq\lambda}K_\alpha(x,y)d\mu(y) &\leq\int_{\lambda\leq\delta(x,y)<R}K_{\alpha}(x,y) d\mu(y)+\alpha\int_{\delta(x,y)\geq R}\frac{d\mu(y)}{\delta(x,y)^{1+s}}\\
&\leq\frac{\alpha}{R^{1+s}} C(R,\lambda)\mu(A(x,\lambda,R))+C\frac{\alpha}{R^{1+s}}.
\end{align*}
For $\alpha\to 0^+$ we get \textit{(i)}, that it $\mathcal{K}\in\mathcal{C}$.

Property \textit{(ii)} follows from Theorem \ref{thm:estimatemaximalHL} and \textit{(iii)} from standard arguments since we are assuming that continuous functions are dense in $L^1(X,\mu)$.
\end{proof}


\providecommand{\bysame}{\leavevmode\hbox to3em{\hrulefill}\thinspace}
\providecommand{\MR}{\relax\ifhmode\unskip\space\fi MR }
\providecommand{\MRhref}[2]{%
  \href{http://www.ams.org/mathscinet-getitem?mr=#1}{#2}
}
\providecommand{\href}[2]{#2}



\bigskip

\bigskip
\noindent{\footnotesize
\textsc{Instituto de Matem\'atica Aplicada del Litoral, UNL, CONICET, FIQ.}

\smallskip
\noindent\textmd{CCT CONICET Santa Fe, Predio ``Alberto Cassano'', Colectora Ruta Nac.~168 km 0, Paraje El Pozo, 3000 Santa Fe, Argentina.}
}
\bigskip

\end{document}